\documentclass[10pt]{amsart}
\usepackage{amssymb}
\usepackage{amsmath}
\usepackage{mathtools}
\usepackage{enumitem}
\usepackage{wrapfig}
\usepackage{url}
\usepackage{graphicx}
\setlist{nolistsep}
\usepackage{bm}
\usepackage[unicode]{hyperref}
\usepackage{tikz}
\usepackage{pgfplots}
\pgfplotsset{compat=1.17}
\usetikzlibrary{patterns,patterns.meta,arrows.meta,positioning}

\definecolor{myblue}{rgb}{0.126, 0.875, 0.588}
\definecolor{mypink}{rgb}{0.875, 0.126, 0.412}

\newcommand{\N}{\mathbb{N}}
\newcommand{\Z}{\mathbb{Z}}
\newcommand{\R}{\mathbb{R}}
\newcommand{\CC}{\mathbb{C}}

\newcommand{\eps}{\varepsilon}

\newcommand{\ucalC}{\underline{\mathcal{C}}}
\newcommand{\Top}[0]{\mathsf{Top}}

\newcommand{\id}[1]{\ensuremath{\text{\normalfont id}_{#1}}}

\makeatletter

\newcommand{\Mod}[1]{\ (\mathrm{mod}\ #1)}

\newcommand{\Ninfty}{N_\infty}

\newcommand{\factor}[2]{\left. \raise 2pt\hbox{\ensuremath{#1}} \right/
 	  \hskip -2pt\raise -2pt\hbox{\ensuremath{#2}}}
    
\newcommand{\Cgp}[1]{\factor{\Z}{#1 \Z}}
\newcommand{\Zgp}[2]{\factor{#1 \Z}{#2 \Z}}

\newtheorem{Thm}{Theorem}[section]

\newtheorem{Prop}[Thm]{Proposition}
\newtheorem{Lemma}[Thm]{Lemma}
\newtheorem{Conjecture}[Thm]{Conjecture}
\newtheorem{Notation}[Thm]{Notation}

\theoremstyle{definition}
\newtheorem{Def}[Thm]{Definition}

\theoremstyle{remark}

\newtheorem{Rmk}[Thm]{Remark}

\usepackage{tikz-cd}
\tikzset{
    weq/.style={anchor=south, rotate=90, inner sep=.5mm}
}

\usetikzlibrary{arrows}
\usetikzlibrary{calc}
\usetikzlibrary{decorations.pathmorphing}

\tikzset{curve/.style={settings={#1},to path={(\tikztostart)
    .. controls ($(\tikztostart)!\pv{pos}!(\tikztotarget)!\pv{height}!270:(\tikztotarget)$)
    and ($(\tikztostart)!1-\pv{pos}!(\tikztotarget)!\pv{height}!270:(\tikztotarget)$)
    .. (\tikztotarget)\tikztonodes}},
    settings/.code={\tikzset{quiver/.cd,#1}
        \def\pv##1{\pgfkeysvalueof{/tikz/quiver/##1}}},
    quiver/.cd,pos/.initial=0.35,height/.initial=0}

\tikzset{tail reversed/.code={\pgfsetarrowsstart{tikzcd to}}}
\tikzset{2tail/.code={\pgfsetarrowsstart{Implies[reversed]}}}
\tikzset{2tail reversed/.code={\pgfsetarrowsstart{Implies}}}
\tikzset{no body/.style={/tikz/dash pattern=on 0 off 1mm}}

\makeatletter
\def\@tocline#1#2#3#4#5#6#7{\relax
  \ifnum #1>\c@tocdepth 
  \else
    \par \addpenalty\@secpenalty\addvspace{#2}%
    \begingroup \hyphenpenalty\@M
    \@ifempty{#4}{%
      \@tempdima\csname r@tocindent\number#1\endcsname\relax
    }{%
      \@tempdima#4\relax
    }%
    \parindent\z@ \leftskip#3\relax \advance\leftskip\@tempdima\relax
    \rightskip\@pnumwidth plus4em \parfillskip-\@pnumwidth
    #5\leavevmode\hskip-\@tempdima
      \ifcase #1
       \or\or \hskip 1em \or \hskip 2em \else \hskip 3em \fi%
      #6\nobreak\relax
    \dotfill\hbox to\@pnumwidth{\@tocpagenum{#7}}\par
    \nobreak
    \endgroup
  \fi}
\makeatother

\tikzset{
    mypinkline/.style={draw=mypink, line width=2pt},
    grayline/.style={draw=gray, line width=2pt},
    dashedline/.style={dashed, draw=black, line width=1pt}
}

\newcommand{\drawbox}[6]{
    \begin{scope}[xshift=#1 cm]
        \ifx#3C \draw[mypinkline] (0,1) -- (1,1); \fi
        \ifx#3G \draw[grayline] (0,1) -- (1,1); \fi
        \ifx#3D \draw[dashedline] (0,1) -- (1,1); \fi
        \ifx#4C \draw[mypinkline] (1,1) -- (1,0); \fi
        \ifx#4G \draw[grayline] (1,1) -- (1,0); \fi
        \ifx#4D \draw[dashedline] (1,1) -- (1,0); \fi
        \ifx#5C \draw[mypinkline] (1,0) -- (0,0); \fi
        \ifx#5G \draw[grayline] (1,0) -- (0,0); \fi
        \ifx#5D \draw[dashedline] (1,0) -- (0,0); \fi
        \ifx#6C \draw[mypinkline] (0,0) -- (0,1); \fi
        \ifx#6G \draw[grayline] (0,0) -- (0,1); \fi
        \ifx#6D \draw[dashedline] (0,0) -- (0,1); \fi
        \node at (0.5,-0.4) {#2};
    \end{scope}
}

\begin{document}

\title[Rubin's saturation conjecture for cyclic groups of order $p^nq^m$]{Realization of saturated transfer systems on cyclic groups of order $p^nq^m$ by linear isometries $\Ninfty$-operads}
\author{Julie Bannwart}
\address{\parbox{\linewidth}{Research conducted at EPFL, Lausanne, Switzerland.\\
Current affiliation of the author: Institut für Mathematik, JGU Mainz, Germany.\\}}
\email{bannwart.julie@gmail.com}
\date{May 2025}

\begin{abstract}
We prove a specific case of Rubin's saturation conjecture about the realization of $G$-transfer systems, for $G$ a finite cyclic group, by linear isometries $\Ninfty$-operads, namely the case of cyclic groups of order $p^nq^m$ for $p,q$ distinct primes and $n,m\in \N$. 
\end{abstract}

\maketitle

\tableofcontents

\section{Introduction}

Blumberg and Hill study in \cite{BlumbergHill} the question of generalizing the theory of $E_\infty$-operads to the $G$-equivariant setting, for a finite group $G$. Instead of considering $E_\infty$-operads in categories of $G$-objects, their approach is
to take into account the $G$-action at a finer level, by considering operads in $G$-spaces, with contractibility conditions on the subspaces of fixed points in the space of operations in arity $n$, for all subgroups of $G\times S_n$, for $S_n$ the symmetric group (in particular, they must be either contractible or empty). The concept obtained is that of $\Ninfty$-operad, and represents (a whole range of) intermediary notion(s) between $E_\infty$-operads in a category and $E_\infty$-operads in its category of $G$-objects. 

A morphism between such operads is a weak equivalence if it induces weak equivalences of topological spaces in all arities, on all subspaces of fixed points for subgroups of $G\times S_n$.

An $\Ninfty$-operad imposes on its algebras the existence of additional structure: contractible spaces of transfer maps, and the information of which transfers exist is parametrized by the data of which spaces of fixed points in the operad are contractible, rather than empty.

These $\Ninfty$-operads are classified up to homotopy by posets refining the poset of subgroups of $G$ (``sub-posets'') satisfying additional conditions (namely being closed under conjugation and restriction), the so-called \emph{$G$-transfer systems}.
\newtheorem*{firstrep}{\emph{\textbf{Theorem \ref{Prop:EquivWithTrSyst}}}}
\begin{firstrep}[\protect{\cite[3.24]{BlumbergHill}}; and \cite{BonventrePereira}, \cite{GutierrezWhite}, \cite{RubinCombinatorial}]
\emph{The homotopy category of $\Ninfty$-operads is equivalent to the poset of $G$-transfer systems ordered by refinement.}
\end{firstrep}

In their foundational article, Blumberg and Hill defined a functor between the homotopy category of $\Ninfty$-operads and another poset, namely that of \emph{$G$-indexing systems}, and showed its fullness and faithfulness, while (essential) surjectivity was only proved some years later (Bonventre \& Pereira in \cite{BonventrePereira}, Gutiérrez \& White in \cite{GutierrezWhite}, Rubin in \cite{RubinCombinatorial}). The poset of $G$-indexing systems was later replaced in the statement by the (equivalent) poset of $G$-transfer systems by Rubin (e.g.\ \cite{RubinDetecting}) and Balchin–Barnes–Roitzheim (\cite{BBR}). In Section \ref{Sub:NInfinity}, we provide a brief introduction to $\Ninfty$-operads and their classification.\\

A particular class of $\Ninfty$-operads is provided by linear isometries operads. Given a $G$-universe $\mathcal{U}$, which is a countably infinite-dimensional real representation of $G$ by linear isometries, with all sub-representations occurring infinitely many times (and the trivial one-dimensional representation must appear), there is an $\Ninfty$-operad $\mathcal{L}(\mathcal{U})$ associated with it, whose operations in arity $n$ consist in the space of isometries $\mathcal{U}^{\oplus{n}} \to \mathcal{U}$. A realization problem then arises: which $G$-transfer systems can be realized by linear isometries operads, as the $G$-universe varies? For some justification as to why one might be interested in answering this question, see for example \cite[\S 1]{Macbrough}. Blumberg and Hill found a necessary condition: the transfer system must be saturated (if it contains the relation $H\subseteq K$ for some subgroups $H, K \leq G$, then for any intermediate subgroup $H \subseteq M \subseteq K$, the transfer system must also contain the relations $H \subseteq M$ and $M \subseteq K$). This condition may however not be sufficient, as proved in \cite{RubinDetecting}, and as we recall in Remark \ref{Rmk:SaturationNotSufficient}.

However, in the case of finite cyclic groups, Rubin conjectured in \cite{RubinDetecting} that the saturation condition becomes sufficient.
\newtheorem*{secrep}{\emph{\textbf{Conjecture \ref{Conj:Saturation}}}}
\begin{secrep}[Rubin's saturation conjecture]
\emph{Let $k\in \N^\ast = \{1,2,\dots\}$ and $e_1,\dots,e_k\in\N^\ast$. There exist integers $p_1,\dots,p_k$ depending on this choice, such that for all $k$-uples of distinct primes $q_1,\dots,q_k$ with $q_i\geq p_i$ for all $i\leq k$, and $G:= C_{q_1^{e_1}\cdots q_k^{e_k}}$, any saturated $G$-transfer system is realized by some linear isometries operad.}
\end{secrep}

The conjecture was proved in \cite{RubinDetecting} for cyclic groups of order $p^n$ and $pq$, and in \cite{SatLinIsomTrSyst} for cyclic groups of order $qp^n$, for $p,q \geq 5$ distinct primes and any $n\in\N$. A few days after the publication of this paper on ArXiv, MacBrough published a preprint (\cite{Macbrough}) giving a positive answer to the conjecture, with $p_i = 5$ for all $i\leq k$ in any situation. His results are more general, and provide in particular a criterion for abelian groups admitting a presentation with two generators to satisfy the saturation conjecture (\cite[Thm 3.14]{Macbrough}).\\

In contrast, our methods are specific to the case of cyclic groups of order $p^nq^m$. Our proof of this particular case is more explicit and enlightens the complexity of the combinatorics involved in these saturated transfer systems, already in the simple case we are interested in. MacBrough writes that in \cite[\S 3.2]{Macbrough} that his proof of the saturation conjecture ``completely bypass[es] all of the diffculties in the direct approach to the saturation conjecture for cyclic groups''. What we do in this article is exactly to try this direct approach in the case of cyclic groups of order $p^nq^m$. Moreover, the proof we give in this case could in principle be applied to any particular example to construct the explicit linear isometries operad associated with a fixed transfer system. Section \ref{Sub:example} contains an example in the case $p=5$, $q=7$ and $n=m=1$, together with other examples and a comparison of MacBrough's approach to ours.\\

For arbitrary cyclic groups, the problem posed by the conjecture was reduced to a purely arithmetic one by Rubin in \cite{RubinDetecting}: there is a characterization of the relations contained in the transfer system arising from $\mathcal{L}(\mathcal{U})$, in terms of the translations under which a certain subset of $C_{n}$ is invariant. This subset is called the \emph{indexing set} and uniquely characterizes the universe $\mathcal{U}$ we are working with. We discuss this conjecture, and the reduction of the problem to modular arithmetic in Section~\ref{Sub:SaturationConjecture}.\\

The proofs of the three cases of the conjecture mentioned above consist in building suitably invariant indexing sets. Following the same approach, and using the previously proven cases as a basis for induction, we prove in Section \ref{Sub:Proof} our main result, i.e.\ the following instances of the saturation conjecture:

\newtheorem*{thirdrep}{\emph{\textbf{Theorem \ref{Prop:SatConjp^nq^m}}}}
\begin{thirdrep}[Saturation conjecture for $p^nq^m$]
\emph{Let $p,q\geq 5$ be distinct primes, and $n,m\in \N$. If $G = C_{p^nq^m}$, then any saturated $G$-transfer system is realized by some linear isometries operad on a $G$-universe $\mathcal{U}$.}
\end{thirdrep}

This result does not hold for $p\leq 3$ or $q\leq 3$, see Remark \ref{Rmk:SaturationNotSufficient} for a counterexample. Saturated transfer systems on $C_{p^nq^m}$ were enumerated in \cite{SatLinIsomTrSyst}, by describing them in terms of \emph{saturated covers} of $[m]\times [n]$ and \emph{compatible codes}.\\

Generalizing our approach to cyclic groups whose order has more than two prime divisors (say, $n$) appears to be quite hopeless in practice, although an inductive argument seems theoretically possible. Indeed, just as in the proofs of Lemmata \ref{Prop:BaseCaseProof} and \ref{Prop:InductionStepProof}, one would probably have to distinguish between the different possibilities for the ``topmost cube'' (with corners $(e_1,\dots,e_n) \in [m_1]\times\dots\times[m_n]$, where $e_i \in \{m_i-1,m_i\}$ for all $1\leq i\leq n$), which has dimension $n$. However, the number of possibilities and the number of constraints for the indexing set to agree with its various restrictions (to allow induction) increase rapidly, and quickly make the construction of explicit indexing sets with good translation properties unmanageable in practice.

\subsubsection*{\normalfont \textbf{Acknowledgements}} The work presented in this paper was carried out during a summer internship in the Laboratory for Topology and Neuroscience at EPFL. I would like to thank the EPFL ``Summer in the Lab'' and ``Student Support'' programs for making this possible and the ``Domaine de Villette'' foundation for their support to the programs. I also would like to express all my gratitude to K. Hess Bellwald and J. Scherer for their continued encouragement and kind guidance. Finally, I thank the anonymous referees for helpful feedback. 

This version of the article has been accepted for publication in the Journal of Homotopy and Related Structures, after peer review, but is not the Version of Record and does not reflect post-acceptance improvements. The Version of Record is available online at: \href{https://doi.org/10.1007/s40062-025-00377-6}{https://doi.org/10.1007/s40062-025-00377-6}.

\section{\texorpdfstring{$G$}{G}-equivariant analog to \texorpdfstring{$E_\infty$}{E-infinity}-operads : \texorpdfstring{$\Ninfty$}{N-infinity}-operads}\label{Sub:NInfinity} 

In this section, we give the formal definition of the equivariant operads advertised in the introduction and present their classification up to homotopy
by transfer systems. Let $G$ be a fixed finite (discrete) group.
\begin{Def} A \emph{$G$-operad} is a topological operad $\mathcal{O}$ (or an operad in $G$-spaces), such that $\mathcal{O}(n)$ is a $(G\times S_n)$-space for all $n\in \N$, with a $G$-fixed identity element and $G$-equivariant structure maps. An \emph{$\Ninfty$-operad} is a $G$-operad such that the following conditions hold:
\begin{itemize}[leftmargin=*]
    \itemsep0em
    \item For all $n\in\N$, the action of the symmetric group $S_n$ on $\mathcal{O}(n)$ is free.
    \item For all subgroups $\Gamma\leq G\times S_n$, the space $\mathcal{O}(n)^\Gamma$ of $\Gamma$-fixed points is either empty or contractible (as a topological space, not necessarily equivariantly).
    \item The space of fixed points $\mathcal{O}(n)^G$ is non-empty for all $n\in\N$.
\end{itemize}
\end{Def}

\subsection{Algebras over $\Ninfty$-operads}\label{Sub:algebras} We clarify what we mean by algebras over such operads. Let $\Top$ be the category of topological spaces, and ${}_G\Top$ the category of $G$-spaces.
\begin{Def}\label{Def:NInfinityAlgebras} Given $\mathcal{O}$ an $\Ninfty$-operad, an \emph{algebra in ${}_G\Top$ over $\mathcal{O}$} is an algebra over the underlying operad in $G$-spaces.
\end{Def}

In fact, $\Ninfty$-algebras can be defined in any symmetric monoidal category that is tensored over the category of $G$-spaces. In \cite{BlumbergHill} the case of orthogonal $G$-spectra is studied in detail. We shall work only with ($G$-)spaces in this article.\\

The axioms in the definition of $\Ninfty$-operads ensure the existence of additional structure on their algebras, namely \emph{transfers maps} (or \emph{norm maps} in the case of spectra). We will state this result later (see Theorem \ref{Prop:StcNInfinityAlgebras}), once we have the language to describe precisely which transfer maps exist.\\

Weak equivalences between $\Ninfty$-operads should also respect the finer structure of all spaces of fixed points. Otherwise, they would amount only to weak equivalences of the underlying topological $E_\infty$-operads. But all categories of algebras over such operads are equivalent, whereas for different $\Ninfty$-operads the algebras do not look quite the same, due to the transfer maps. The richer structure on $\Ninfty$-operads will become apparent in their classification up to homotopy in Theorem \ref{Prop:EquivWithTrSyst}.
\begin{Def}\label{Def:WeqNInfinityOperads} A morphism of $\Ninfty$-operads $f:\mathcal{O} \to \mathcal{O}'$ is called a \emph{weak equivalence} if the underlying map of $G$-operads is a weak equivalence, namely if the induced $G$-equivariant map $f^\Gamma:\mathcal{O}(n)^\Gamma \to \mathcal{O}'(n)^\Gamma$ is a (non-equivariant) weak homotopy equivalence on the underlying topological spaces for all $n\in\N$ and all subgroups $\Gamma\leq G\times S_n$. By localizing with respect to these maps (as a category with weak equivalences), we obtain the \emph{homotopy category} of $\Ninfty$-operads, denoted by $\text{Ho}(\Ninfty\text{-Op})$.
\end{Def}
Definition \ref{Def:WeqNInfinityOperads} is justified by the fact that weakly equivalent $\Ninfty$-operads have Quillen equivalent categories of algebras, provided that $\mathcal{O}(n)$ and $\mathcal{O}'(n)$ are nice enough spaces for all $n\in\N$ (see \cite[Thm A.3]{BlumbergHill}).

\subsection{Linear isometries operads}

To define this class of $\Ninfty$-operads we are particularly interested in, we need specific representations of $G$.
\begin{Def} A \emph{$G$-universe} $\mathcal{U}$ is a real vector space of countably infinite dimension, endowed with an inner product and an action of $G$ by linear isometries, such that any sub-representation occurs infinitely often, including the trivial one.
\end{Def}

Since $G$ is finite, any $G$-universe can be written as $\bigoplus_{\N} (\R_\text{triv} \oplus V_1 \oplus \dots \oplus V_k)$ for some finite dimensional (irreducible) real linear isometric representations $V_1, \dots, V_k$ of $G$, with $\R_\text{triv}$ the trivial representation. 

\begin{Def} Given $\mathcal{U}$ a $G$-universe, the \emph{linear isometries operad} $\mathcal{L}(\mathcal{U})$ is the topological operad given in arity $n\in\N$ by the space of (non necessarily $G$-equivariant) isometries $\mathcal{U}^{\oplus n} \to \mathcal{U}$. The (left-) action of $S_n$ and the composition maps are defined as in the usual endomorphism operad, with the unit being the identity on $\mathcal{U}$. Given an isometry $f: \mathcal{U}^{\oplus n} \to \mathcal{U}$, $g\in G$ and $(\vec{x}_1,\dots,\vec{x}_n)\in\mathcal{U}^{\oplus n}$, let $(g\cdot f)(\vec{x}_1,\dots,\vec{x}_n) = g\cdot f(g^{-1}\cdot\vec{x}_1,\dots,g^{-1}\cdot\vec{x}_n)$ (action by conjugation).
\end{Def}

In particular, a $G$-fixed point in $\mathcal{L}(\mathcal{U})(n)$ is just a $G$-equivariant linear isometry $\mathcal{U}^{\oplus n}\to \mathcal{U}$.

Other examples of $\Ninfty$-operads include the \emph{equivariant infinite little disks operads}, the \emph{Steiner operads} or the \emph{embeddings operads}, all depending on a $G$-universe (see \cite[3.11]{BlumbergHill}).

\subsection{Classification of $\Ninfty$-operads up to homotopy}

Blumberg and Hill provide in \cite{BlumbergHill} a beautiful classification of $\Ninfty$-operads up to homotopy: their homotopy category is equivalent to a specific poset. To define the latter, we need to define transfer systems. These objects are due, independently, to Rubin (e.g.\ \cite{RubinDetecting}) and Balchin–Barnes–Roitzheim (\cite{BBR}).

\begin{Def}\label{Def:GTrSyst} A \emph{$G$-transfer system} is a relation $\to$ refining the inclusion $\leq$ in the lattice of subgroups of $G$, with, for all $K \leq H\leq G$:
\begin{itemize}[leftmargin=*]
    \itemsep0ex
    \item Reflexivity: $H \to H$.
    \item Transitivity/self-induction: if $K\to M$ and $M \to H$ for some $M\leq G$, then $K \to H$.
    \item Closure under conjugation: if $K\to H$, then $(gKg^{-1})\to(gHg^{-1})$ for all $g\in G$.
    \item Closure under restriction: $K\to H \implies \forall M\leq H, K\cap M \to M$.
\end{itemize}
\end{Def}

In other terms a $G$-transfer system is a wide subcategory of the poset of subgroups of $G$, closed under conjugation and base change in pullback squares. Transfer systems on $G$ form a poset $\text{Tr}(G)$ with respect to inclusion (refinement).

\begin{Rmk} 
When $G$ is abelian, the conjugation condition holds trivially, so the definition does not use the group structure anymore and becomes purely combinatorial. It can therefore be generalized to any lattice, with intersection replaced by the ``meet'' operation. On finite lattices, transfer systems are in one-to-one correspondence with weak factorization systems (providing the right class of maps) and therefore with contractible model structures (i.e., with all maps being weak equivalences) (see \cite{Mstc[n]}).
\end{Rmk}

We are now ready to state the classification theorem.

\begin{Thm}[\protect{\cite[3.24]{BlumbergHill}} and e.g.\ \cite{RubinCombinatorial}]\label{Prop:EquivWithTrSyst}
There is an equivalence of categories
$$\ucalC : \text{Ho}(\Ninfty\text{-Op})\longrightarrow \text{Tr}(G)$$
between the homotopy category of $\Ninfty$-operads and the poset of $G$-transfer systems.
\end{Thm}
Actually, Blumberg and Hill originally used another poset, the poset of \emph{indexing systems}, to prove that $\ucalC$ was fully faithful, and conjectured it was an equivalence of categories, which was soon proved by several authors in different ways. Gutiérrez and White (\cite{GutierrezWhite}) for instance prove the existence of an $\Ninfty$-operad associated with each suitable sequence of subgroups of $G\times S_n$, using model structures on the category of $G$-operads. Bonventre and Pereira (\cite{BonventrePereira}) use ``genuine equivariant operads'' and bar constructions. The proof by Rubin (\cite{RubinCombinatorial}) is more combinatorial, and uses a ``discrete'' version of $\Ninfty$-operads, namely some particular operads in $G$-sets, whose homotopy theory is equivalent to that of $\Ninfty$-operads.

Using transfer systems instead of indexing systems constitutes an equivalent approach to the problem. The former can be seen as \emph{generating data} for the latter: indexing systems are expressed in terms of categories of $H$-sets when $H$ varies among the subgroups of $G$, and transfer systems correspond to the orbit objects in these categories, i.e., the $H$-sets isomorphic to $H/K$ for $K\leq H$ a subgroup. Transfer systems can be more convenient to work with as they are smaller, with an a priori simpler definition.\\

Let us now describe the functor $\ucalC$. We need preliminary definitions.
\begin{Def} Let $H\leq G$ be a subgroup and $T$ a finite $H$-set. The \emph{graph subgroup associated with $T$} is the (conjugacy class of the) subgroup ${\Gamma_T \leq G \times S_{|T|}}$ given by the graph of the homomorphism $H\to S_{|T|}$ sending $h\in H$ to the permutation $\sigma_h \in S_{|T|}$ such that $h\cdot t_i = t_{\sigma_h(i)}$ for all $i\leq |T|$, for $t_1,\dots,t_{|T|}$ an enumeration of $T$. If $T = H/K$ for some $K\leq H$, we write $\Gamma_T = {\Gamma_{H,K}}$.
\end{Def}

This is an abuse of notation because this subgroup may change depending on the enumeration of $T$ chosen. However, if we relabel $T$ using a permutation $\tau$, then the subgroups obtained are conjugated by $\tau$. Actually, all subgroups of $G\times S_n$ whose intersection with $S_n$ is trivial are graph subgroups, which can be seen as follows. Let $\Gamma \leq G\times S_n$ be such a subgroup, and $H\leq G$ be its projection on the first component. If $h\in H$, then there exists some $\sigma_h\in S_n$ with $(h,\sigma_h)\in \Gamma$. Moreover, $\sigma_h$ is unique with this property: if $(h,\tau) \in\Gamma$, then $(e,\sigma_h\tau^{-1})\in \Gamma\cap S_n = \{(1,\id{})\}$, so $\sigma_h=\tau$. The assignment $h\mapsto \sigma_h$ is a group homomorphism: given $h,h'\in H$, we have $(h,\sigma), (h',\sigma_{h'})\in \Gamma$, so $(hh',\sigma_h\sigma_{h'})\in\Gamma$, and by uniqueness $\sigma_{hh'} = \sigma_h\sigma_{h'}$. Therefore, any such subgroup determines a morphism $H\to S_n$, which yields exactly an $H$-set structure on $\{1,\dots,n\}$.

\begin{Def}\label{Def:Admissible} Let $\mathcal{O}$ be an $\Ninfty$-operad and $H\leq G$ a subgroup. A finite $H$-set $T$ is called \emph{admissible} if $\mathcal{O}(|T|)^{\Gamma_T} \neq \emptyset$. 
\end{Def} 

This is well-defined, because the condition about the fixed points being non-empty depends only on the conjugacy class of $\Gamma_T$.

\begin{Def}\label{Def:FunctorC} The functor $\ucalC$ in Theorem \ref{Prop:EquivWithTrSyst} is induced, using the universal property of localization, by the functor $\ucalC: \Ninfty\text{-Op} \to \text{Tr}(G)$ sending an $\Ninfty$-operad $\mathcal{O}$ to the $G$-transfer system $\to$ such that $K\to H$ if and only if $K\leq H\leq G$ are subgroups, and $H/K$ is admissible as an $H$-set for $\mathcal{O}$. The relations contained in the transfer system are called \emph{admissible}.
\end{Def}

Admissible relations are the ones for which the transfer maps advertised in Subsection \S \ref{Sub:algebras} can be constructed. 

\begin{Thm}[\protect{\cite[7.1, 7.2]{BlumbergHill}}, \protect{\cite[3.5]{RubinDetecting}}]\label{Prop:StcNInfinityAlgebras} Let $\mathcal{O}$ be an $\Ninfty$-operad, and $X$ an algebra in ${}_G\Top$ over $\mathcal{O}$.
\begin{itemize}[leftmargin=*]
    \itemsep0em
    \item Given an admissible relation $K\to H$ (see Definitions \ref{Def:Admissible} and \ref{Def:FunctorC}) for subgroups $K\leq H\leq G$, there is a contractible space of maps $$(G\times S_{[H:K]})/\Gamma_{H,K} \longrightarrow \mathcal{O}([H:K]).$$
    \item Again assuming that $K \to H$ is admissible, there are contractible spaces of \emph{internal transfer maps} of $H$-spaces ${X^K \to X^H}$ and \emph{external transfer maps} of $G$-spaces $G\times_H X^{\times H/K} \to X$.
    \item If $K\leq H$ and $N\leq H$ are admissible, any $H$-equivariant map $H/K\to H/N$ induces a contractible space of $H$-equivariant maps 
    $$\Top\left(H/K,X\right) \longrightarrow \Top\left(H/N,X\right),$$
    where the action of $H$ is by conjugation (and $X$ is viewed as an $H$-space).
\end{itemize}
\end{Thm}

\section{Rubin's saturation conjecture}\label{Sub:SaturationConjecture}

We present in this section Rubin's conjecture on saturated transfer systems for cyclic groups and his description of universes by indexing sets. For any $n\in\N^\ast$ and $m\mid n$, let $C_n := \Cgp{n}$ and $mC_n := \Zgp{m}{n}$. In particular, we use additive notation. Propositions \ref{Prop:AdmissibleSetsLAndD} and \ref{Prop:SaturatedLinIsom} already appear in \cite{BlumbergHill}, we repeat their proofs below only for completeness, and to add details. \\

By Theorem \ref{Prop:EquivWithTrSyst}, any $G$-transfer system is realized by some $\Ninfty$-operad. Blumberg and Hill asked which ones could be realized by a linear isometries operad. The admissible relations for the latter can be characterized as follows.

\begin{Prop}[\protect{\cite[4.18]{BlumbergHill}}]\label{Prop:AdmissibleSetsLAndD} Given $K\leq H\leq G$, $H/K$ is admissible for $\mathcal{L}(\mathcal{U})$ if and only if there exists an $H$-equivariant embedding $\Z [H/K]\otimes \mathcal{U} \longrightarrow \mathcal{U}$.
\end{Prop}

\begin{Rmk}
    The group $H$ acts on the tensor product $\Z\left[H/K\right]\otimes \mathcal{U}$ as follows: if $H/K = \{h_1K, \dots, h_kK\}$, and $h\in H$, $i\leq k$, $u\in\mathcal{U}$, there exists a unique $k_i(h)\in K$ with $hh_i = h_{\sigma_h(i)}k_i(h)$. We set $h\cdot (h_iK)\otimes u = h_{\sigma_h(i)}K \otimes k_i(h)u$. In particular, $\Z\left[H/K\right]\otimes \mathcal{U}$ is isomorphic as a representation of $H$ to $\Z[H]\otimes_{\Z[K]} \mathcal{U}$. 
\end{Rmk}
\begin{proof}
    Let us write $H/K=\{h_1K, \dots, h_nK\}$. A $\Gamma_{H,K}$-fixed point in $\mathcal{L}(\mathcal{U})(n)$ is by definition a $\Gamma_{H,K}$-equivariant map $F:\mathcal{U}^{\oplus n} \to \mathcal{U}$, where elements of the symmetric group permute the variables, and $H$ acts by conjugation. By definition $\Gamma_{H,K}$ consists of elements $(h,\sigma_h)$, where $\sigma_h$ describes the permutation induced on $H/K$ by $h\in H$. In particular, this fixed point gives us an $H$-equivariant embedding under the identification $\Z\left[H/K\right] \otimes \mathcal{U} \cong \mathcal{U}^{\oplus n}$, with $h_iK \otimes u$ sent to $h_i\cdot u$ on the $i$-th summand. Indeed, the map $f:\Z\left[H/K\right] \otimes \mathcal{U} \to \mathcal{U}$ obtained is $H$-equivariant. For all $h\in H$, $i\leq n$ and $u\in \mathcal{U}$, we obtain from the fixed point condition:
    \begin{align*}
    h\smash\cdot (f(h_iK\otimes u)) &= h\smash\cdot (F(0,\dots,h_iu,\dots,0)) \hspace{2.2cm}  \text{($h_iu$ in the $i$-th summand)} \\
    &= h\smash\cdot((h^{-1},\sigma_{h}^{-1})\smash\cdot F)(0,\dots,h_iu,\dots,0) \\
    &= hh^{-1}\smash\cdot (F(0, \dots, hh_iu, \dots, 0)) \hspace{0.45cm} \text{ ($hh_iu$ in the $\sigma_h(i)$-th summand)}\\
    &= f(h_{\sigma_h(i)}K\otimes k_i(h)u) \\
    &= f(h\smash\cdot(h_iK\otimes u)).
    \end{align*}
    The proof of the converse implication is similar (given an $H$-equivariant embedding $f:\Z\left[H/K\right] \otimes \mathcal{U} \to \mathcal{U}$, the corresponding map $\mathcal{U}^{\oplus n} \cong \Z\left[H/K\right] \otimes \mathcal{U}  \to \mathcal{U}$ is a $\Gamma_{H,K}$-fixed point in $\mathcal{L}(\mathcal{U})(n)$).
    \end{proof}

Going back to our realization problem, the following condition is necessary.
\begin{Prop}[\cite{BlumbergHill}] \label{Prop:SaturatedLinIsom} If $G$ is a finite group, then for any $G$-universe $\mathcal{U}$, the transfer system $\ucalC(\mathcal{L}(\mathcal{U}))$ is saturated, i.e., if $K\to H$ is admissible, and $K\leq N\leq H$ is an intermediary subgroup, then both $K\leq N$ and $N\leq H$ are admissible.
\end{Prop}
\begin{proof}
    The admissibility of $K = N \cap K \leq N$ is by restriction of $K\to H$. For the other relation, by Theorem \ref{Prop:AdmissibleSetsLAndD}, we need an $H$-embedding $\Z [H/N]\otimes \mathcal{U} \to \mathcal{U}$. Since $K\leq H$ is admissible, there is an $H$-equivariant embedding ${\Z[H/K]\otimes \mathcal{U} \to \mathcal{U}}$. It therefore suffices to find an $H$-embedding $\Z [H/N]\otimes \mathcal{U} \to \Z [H/K]\otimes \mathcal{U}$, or equivalently, $f: \Z[H]\otimes_{\Z[N]} \mathcal{U} \to \Z[H] \otimes_{\Z[K]} \mathcal{U}$. Write $N/K = \{n_1K, \dots,n_kK\}$. We define $f$ as the linear extension of the assignment ${h\otimes u \mapsto \sum_{i\leq k} hn_i \otimes (n_i)^{-1}u}$ for any $u\in\mathcal{U}$, $h\in H$. This is well-defined since for $n\in N$, if $nn_i = n_{\sigma_n(i)}k_i(n)$ for $k_i(n)\in K$ then:
    \begin{align*}
    f(hn\otimes n^{-1}u) &= \sum_{i\leq k} hnn_i \otimes (n_i)^{-1}n^{-1}u = \sum_{i\leq k} hn_{\sigma_n(i)}k_i(n) \otimes (n_i)^{-1}n^{-1}u \\
    &= \sum_{i\leq k} hn_{\sigma_n(i)} \otimes k_i(n)(nn_i)^{-1}u = \sum_{i\leq k} hn_{\sigma_n(i)} \otimes n_{\sigma_n(i)}^{-1}u \\
    &= f(h\otimes u).
    \end{align*}
    This is also $H$-equivariant since, for all $h,h'\in H$ and $u\in\mathcal{U}$:
    $$f(h\cdot(h'\otimes u)) = \sum_{i\leq k} hh'n_i \otimes (n_i)^{-1}u = h\cdot \left(\sum_{i\leq k} h'n_i \otimes (n_i)^{-1}u\right) = h\cdot f(h'\otimes u).$$
    The post-composition of $f$ by the map $\Z[H]\otimes_{\Z[K]} \mathcal{U} \to \Z[H] \otimes_{\Z[N]} \mathcal{U}$ sending $h\otimes_{\Z [K]} u \mapsto h\otimes_{\Z [N]} u$ is the map $[N:K]\cdot\id{}$, which is injective. Therefore, $f$ is an embedding.
\end{proof}

Other necessary conditions and characterizations are proved in \cite{RubinDetecting}, also for the transfer systems arising from Steiner operads.\\

Rubin conjectured the following:
\begin{Conjecture}[Rubin's saturation conjecture (see \cite{RubinDetecting})]\label{Conj:Saturation} Let $k\in \N^\ast$ and $e_1,\dots,e_k\in\N^\ast$. There exist integers $p_1,\dots,p_k$ depending on this choice such that for all $k$-uples of distinct primes $q_1,\dots,q_k$ with $q_i\geq p_i$ for all $i\leq k$, any saturated $(C_{q_1^{e_1}\cdots q_k^{e_k}})$-transfer system is realized by some linear isometries operad. In this case, we say that the group $C_{q_1^{e_1}\cdots q_k^{e_k}}$ satisfies the saturation conjecture.
\end{Conjecture}

\begin{Notation}\label{Notation:subgpCn} Let $n$ be a positive integer with prime decomposition $p_1^{e_1}\dots p_k^{e_k}$. Then, we identify the poset of subgroups of $C_{n}$ with the product $[e_1]\times \cdots\times [e_k]$, where $[e_i]$ denotes the poset $\{0 < 1 < \dots < e_i\}$. Under this identification, the subgroup $p^{f_1}C_{p^{e_1}}\times\cdots\times p^{f_k}C_{p^{e_k}} \cong C_{p^{e_1-f_1}}\times \cdots \times C_{p^{e_k-f_k}}$ corresponds to $(e_1-f_1,\dots,e_k-f_k)$.
\end{Notation}

The number of $C_n$-universes grows exponentially with $n\in\N$. Indeed, there are $2^{\lfloor n/2 \rfloor}$ non-isomorphic $C_n$-universes (by Proposition \ref{Prop:IndexingSets} below), whereas the number of transfer systems is fixed if we fix the number of primes factors of $n$ and their exponents, but not the primes themselves. Therefore, when the $C_n$-universe varies, many linear isometries operads give rise to the same transfer system, and are therefore equivalent.\\

We repeat here that MacBrough has recently provided in \cite{Macbrough} an answer to Conjecture \ref{Conj:Saturation}. The particular cases of cyclic groups of order $p^n$ and $pq$ were proved in \cite{RubinDetecting}, and that of cyclic groups of order $qp^n$ in \cite{SatLinIsomTrSyst}, with $p,q \geq 5$ distinct primes and $n\in\N$ arbitrary. In the same paper, an explicit formula for the number of saturated transfer systems on $C_{p^nq^m}$ is computed.
Proposition \ref{Prop:IndexingSets} below, proved by Rubin, reduces the problem to an arithmetic one.

\begin{Notation} For $n\in \N^\ast$ and $0\leq j \leq n-1$, let $\lambda_n(j)$ be the two-dimensional real representation of $C_n$, where $[1]$ acts by multiplication by $e^{\frac{2\pi i j}{n}}$ in the complex plane. 
\end{Notation}
\begin{Def}\label{Def:indexing} Let $n\in\N^\ast$. An \emph{indexing set} for $C_n$ is a subset $I \subseteq C_n$ such that $0\in I$ and $-I \subseteq I$, where $-I := \{n-i \mid i \in I\}$. For each indexing set, we can define an \emph{associated $C_n$-universe} $\mathcal{U}_I := \bigoplus_{n\in \N} \bigoplus_{j\in I} \lambda_n(j)$. We say that $\mathcal{U}_I$ \emph{realizes} the associated transfer system $\ucalC(\mathcal{L}(\mathcal{U}_I))$.
\end{Def}

\begin{Prop}[\protect{\cite[5.14, 5.15]{RubinDetecting}}]\label{Prop:IndexingSets} Let $G = C_{n}$ for $n\in \N^\ast$.
\begin{enumerate}[label=(\roman*), leftmargin=*]
    \itemsep0ex
    \item Any $G$-universe is of the form $\mathcal{U}_I$ for some indexing set $I$.
    \item The relation $C_{d} \cong (n/d)C_n \to (n/e)C_n \cong C_{e}$ for $d\mid e \mid n$ is admissible in $\ucalC(\mathcal{L}(\mathcal{U}_I))$ if and only if $(I \Mod{e})+d = I \Mod{e}$ (in particular it suffices to check that $(I \Mod{e})+d \subseteq I \Mod{e}$). 
\end{enumerate}
\end{Prop}

The proofs of the special cases of the saturation conjecture mentioned above consist in building explicitly an indexing set realizing any given saturated transfer system. Applying the same method, we prove in the next section the conjecture in the case of groups of the form $C_{p^nq^m}$ with $p,q\geq 5$ distinct primes.

\begin{Rmk}\label{Rmk:SaturationNotSufficient} If either $p\leq 3$ or $q\leq 3$ (and $n,m\geq 1$), there are saturated transfer systems that are not realized by any linear isometries operad. Indeed, assume $p\leq 3$ and $n,m\geq 1$, and consider the saturated transfer system on $[n]\times [m]$ (see Notation \ref{Notation:subgpCn}) consisting of the single map $(0,0) \to (0,1)$ (i.e., $\{0\} \to C_{q}$ in terms of subgroups) and the identities. It is not realized by any linear isometries operad, by the following argument. If $I\subseteq C_{p^nq^m}$ was an indexing set realizing it, then $I\Mod{pq} := J$ would realize the restriction of the transfer system to $[1]\times [1]$, because, by Proposition \ref{Prop:IndexingSets}, its admissible relations are characterized by translation invariance properties modulo $p$, $q$ or $pq$, which are just the same for $I$. This is impossible, as proved in \cite[5.22]{RubinDetecting}. Indeed, this would imply that $J\subseteq pC_{pq}$: else, ${J\Mod{p}\neq \{0\}}$, but there are only one or two (inverse to one another) non-trivial element(s) modulo $p$ if $p\leq 3$, so $J \Mod{p}= C_{p}$. Therefore, $J \Mod{p}$ would be invariant by translation by 1, so $\{0\} \to C_{p}$ would be admissible. Now, since $\{0\} \to C_{q}$ is admissible, $J\Mod{q}$ is invariant by translation by $1$, so $J \Mod{q} = C_{q}$. Therefore, $J = pC_{pq}$ (if $|J|<q$ then also $|J \Mod{q}|<q$). But then $J$ is invariant by translation by $p$, and so $C_{p} \to C_{pq}$ is admissible (i.e., $(1,0) \to (1,1)$) which is a contradiction. The case $q\leq 3$ is symmetric to the one we just considered.
\end{Rmk}

\section{Proof of the main Theorem}\label{Sub:Proof}
\addtocontents{toc}{\protect\setcounter{tocdepth}{1}}

In this section, we prove our main result, namely the following theorem.

\begin{Thm}[Saturation conjecture for $p^nq^m$]\label{Prop:SatConjp^nq^m} Let $p,q\geq 5$ be distinct primes, and $n,m\in \N$. Let $G = C_{p^nq^m}$. Then, any saturated $G$-transfer system is realized by some linear isometries operad on a $G$-universe $\mathcal{U}$.
\end{Thm}

Our proof is quite technical. Even if it does not bring essentially new ideas, it highlights the increasing complexity of the combinatorics one needs to understand when moving from the case $p^nq$ to $p^nq^m$.

\subsection{Outline of the proof}
We fix $p,q\geq 5$ distinct primes, and proceed by induction on $n$ and $m$. We separate base case and induction step in two lemmata. Recall the identification of the poset of subgroups of $C_{p^nq^m}$ with $[n]\times [m]$ in Notation \ref{Notation:subgpCn}.

\begin{Lemma}\label{Prop:BaseCaseProof}
    Let $m\geq 1$, and let $\mathcal{T}$ be a saturated $C_{q^{m+1}}$-transfer system. Given any indexing set $J\subseteq C_{q^m}$ realizing the restriction of $\mathcal{T}$ to $[m]$, there exists an indexing set $I \subseteq C_{q^{m+1}}$ realizing $\mathcal{T}$, with $I \Mod{q^m} = J$, and containing $q^m$. 
    
    Any saturated $C_{q^{m+1}}$-transfer system can be realized by an indexing set containing $q^i$, for all $0 \leq i\leq m$.
\end{Lemma}

\begin{Notation} Let $n,m\in \N$. We say that an indexing set $I \subseteq C_{p^nq^m}$ satisfies $(\star)$ if $I \Mod{p^{i+1}q^m}$ contains a non-zero multiple of $p^iq^m$ for all $0\leq i\leq n-1$.
\end{Notation}

\begin{Lemma}\label{Prop:InductionStepProof} Let $n,m\in \N$. Consider a saturated $C_{p^{n}q^{m+1}}$-transfer system $\mathcal{T}$, and an indexing set $J\subseteq C_{p^{n}q^{m}}$ realizing the restriction of $\mathcal{T}$ to $C_{p^{n}q^{m}}$, and satisfying $(\star)$. Then, there exists an indexing set $K\subseteq C_{p^{n}q^{m+1}}$ realizing $\mathcal{T}$ and satisfying $(\star)$, such that $K$ contains a non-zero multiple of $p^nq^m$ and $K \Mod{p^nq^m} = J$.
\end{Lemma}
We can imagine the situation as follows, with the large dots representing the non-zero multiples required in the different reductions of our indexing sets.

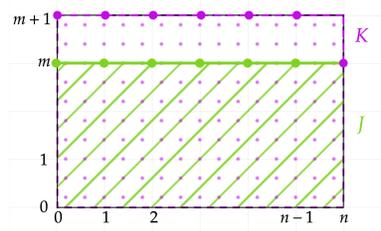
\begin{center}
\begin{tikzpicture}[scale=0.8]

 \draw[pattern={Lines[angle=-45,distance={12pt}]},pattern color=myblue] (0,0) rectangle (6,3);

 \draw[pattern={Dots[angle=45,distance={6pt}]},pattern color=mypink] (0,0) rectangle (6,4);

\draw[very thick] (0,0) rectangle (6,4);

\draw [gray!10, step=1.0cm] (-1,-1) grid (7,5);

\draw[very thick, myblue] (0,3) -- (6,3);
\foreach \x in {0,...,5} {
    \filldraw[fill=myblue, draw=myblue] (\x,3) circle (2pt);
}

\draw[very thick, mypink, dashed] (0,0) rectangle (6,4);
\foreach \x in {0,...,5} {
    \filldraw[fill=mypink, draw=mypink] (\x,4) circle (2pt);
}
\filldraw[fill=mypink, draw=mypink] (6,3) circle (2pt);

\node at (-0.4,0) {\small $0$};
\node at (-0.4,1) {\small $1$};
\node at (-0.4,3) {\small $m$};
\node at (-0.7,4) {\small $m+1$};
\node at (0,-0.4) {\small $0$};
\node at (1,-0.4) {\small $1$};
\node at (2,-0.4) {\small $2$};
\node at (5,-0.4) {\small $n - 1$};
\node at (6,-0.4) {\small $n$};

\node[mypink] at (6.5,3.5) {\small $K$};
\node[myblue] at (6.5,1.5) {\small $J$};

\end{tikzpicture}
\end{center}

Let us now see how this implies our result.

\begin{proof}[Proof of Theorem \ref{Prop:SatConjp^nq^m}] We will prove by induction on $m$ a slightly stronger statement, namely:
\begin{center}
    Any saturated transfer system $\mathcal{T}$ on $C_{p^nq^m}$, with $n\geq 1$ and $m\geq 0$,\\
    can be realized by an indexing set fulfilling $(\star)$.
\end{center}

Together with Lemma \ref{Prop:BaseCaseProof} for the case $n=0$, this will prove Theorem \ref{Prop:SatConjp^nq^m}.

To prove the stronger statement, fix $n\geq 1$ an integer. We proceed by induction on $m$. For $m=0$, the claim follows directly from the second part of Lemma \ref{Prop:BaseCaseProof}. Assume now our claim is true for some $m\in \N$. Taking $\mathcal{T}$ a saturated transfer system on $C_{p^nq^{m+1}}$, by the induction hypothesis its restriction to $[n]\times [m]$ can be realized by some indexing set $J$ satisfying $(\star)$. By applying Lemma \ref{Prop:InductionStepProof} we get an indexing set realizing $\mathcal{T}$, and satisfying $(\star)$, as desired.
\end{proof}

\subsection{Proof of Lemma \ref{Prop:BaseCaseProof}}
We recall once more that the saturation conjecture has already been proved by Rubin in \cite{RubinDetecting} for groups of the form $C_{q^m}$ with $m\in\N$ and $q$ any prime, by exhibiting an explicit indexing set. The crucial observation is that a saturated transfer system on $[m]$ is uniquely determined by its admissible \emph{cover} relations: by transitivity and saturation, a relation $i\to j$ with $i<j$ is admissible if and only if all relations $k\to k+1$ with $i\leq k \leq j-1$ are admissible. We prove this case again, assuming $q\geq 5$, because we need the more precise statement of Lemma \ref{Prop:BaseCaseProof} for the remainder of the proof.

\begin{proof} Let $m$, $\mathcal{T}$ and $J$ be as in the statement. We identify $J$ as a subset of $\{0,1,\dots,q^m -1\}$. We first assume that $m\to m+1$ is admissible in $\mathcal{T}$. Similarly to the proof of \cite[5.18]{RubinDetecting}, we define 
$$I = \{ \pm (j+\alpha q^m) \mid j\in J, 0\leq \alpha < q\}\subseteq C_{q^{m+1}}.$$
Then $I$ is an indexing set restricting to $J$: indeed, $0\in I$, $-I = I$, $I$ contains $q^m$ (since $0\in J$), and $I \Mod{q^m} = \{ \pm j \mid j\in J\} = J$. In particular, it suffices to check that $I$ admits $m\to m+1$ to show that $I$ realizes~$\mathcal{T}$: indeed the other cover relations admissible for $I$ match those of $\mathcal{T}$, because by Proposition \ref{Prop:IndexingSets} this only depends on $I\Mod{q^m} = J$ and $J$ realizes the restriction of $\mathcal{T}$ to $[m]$. This last cover relation is admissible in $I$, since $I + q^m \Mod{q^{m+1}} \subseteq I$ by construction (we can always replace $\alpha$ by its residue modulo $q$ since we work modulo $q^{m+1}$).

Now, if $m\to m+1$ is not admissible in $\mathcal{T}$, consider instead the indexing set $I = \{ \pm (j+\eps q^m) \mid j\in J, \eps\in\{0,1\}\}\subseteq C_{q^{m+1}}$. Then as before $I$ contains $q^m$, and $I \Mod{q^m} = J$. Hence, to show that $I$ realizes $\mathcal{T}$, we only have to check that $m\to m+1$ is not admissible in $I$, i.e., $I$ is not $q^m$-translation invariant. And indeed, $2q^m\notin I$: we have $0 < 2q^m < q^{m+1}$ since $q\geq 5$. And for all $j\in J$, $\eps\in\{0,1\}$, we have $q^{m+1}-(j+\eps q^m) > q^{m+1}-2q^m = (q-2)q^m > 2q^m$ (since $q\geq 5$) and $j + \eps q^m < q^m + q^m = 2q^m$. So $I$ realizes $\mathcal{T}$, as desired.

For the second part of the statement, just begin with the trivial indexing set in $C_1$, and then use inductively the first part of the claim to extend it to an indexing set realizing the restriction of $\mathcal{T}$ to $C_q$ first, and then $C_{q^2}$ and so on. By the way we constructed these extensions above, each successive indexing set contains the previous one, in particular it contains the required powers of $q$. 
\end{proof}

\subsection{Proof of Lemma \ref{Prop:InductionStepProof}}\label{Sub:cases}
\begin{proof} We take our inspiration from the proof of the case $C_{pq^m}$ in \cite{SatLinIsomTrSyst}. We proceed by induction on $n$. The case $n=0$ follows directly from Lemma \ref{Prop:BaseCaseProof}. Assume now that the statement holds for some fixed $n-1\in \N$ and every $m\in \N$. Let us show that it holds for $n$ and every $m\in \N$. As in the statement, consider a saturated $C_{p^{n}q^{m+1}}$-transfer system $\mathcal{T}$, and an indexing set $J\subseteq C_{p^{n}q^{m}}$ satisfying $(\star)$ and realizing the restriction of $\mathcal{T}$ to $C_{p^{n}q^{m}}$. Let $\mathcal{T}'$ be the restriction of $\mathcal{T}$ to $[n-1]\times [m+1]$. Then the restriction of $\mathcal{T}'$ to $[n-1]\times [m]$ is realized by $J' := J \Mod{p^{n-1}q^m} \subseteq C_{p^{n-1}q^m}$, and $J' \Mod{p^{i+1}q^m} = J \Mod{p^{i+1}q^m}$ contains a non-zero multiple of $p^iq^m$ for all $0\leq i < n-1$, so $J'$ satisfies $(\star)$ as well. By our induction hypothesis on $n$, we may therefore find an indexing set $I\subseteq C_{p^{n-1}q^{m+1}}$ satisfying $(\star)$, realizing $\mathcal{T}'$, containing a non-zero multiple of $p^{n-1}q^m$, and such that $I\Mod{p^{n-1}q^m} = J'$. We illustrate the situation as follows:

\begin{center}
\begin{tikzpicture}[scale=0.8]

\foreach \x in {0,...,6} {
    \foreach \y in {0,...,4} {
        \fill[black!50] (\x,\y) circle (1pt);
    }
}

 \draw[pattern={Dots[angle=45,distance={6pt}]},pattern color=mypink] (0,0) rectangle (6,4);

 \draw[pattern={Lines[angle=-45,distance={12pt}]},pattern color=myblue] (0,0) rectangle (6,3);

 \draw [gray!10, step=1.0cm] (-1,-1) grid (7,5);

\draw[very thick] (0,0) rectangle (6,4);
\draw[very thick, myblue, dashed] (5,0) -- (5,3);
\draw[very thick, mypink, dashed] (5,3) -- (5,4);

\draw[very thick, myblue] (0,3) -- (6,3);
\foreach \x in {0,...,5} {
    \filldraw[fill=myblue, draw=myblue] (\x,3) circle (2pt);
}

\draw[very thick, mypink, dashed] (0,4) -- (5,4);
\foreach \x in {0,...,4} {
    \filldraw[fill=mypink, draw=mypink] (\x,4) circle (2pt);
}

\filldraw[fill=myblue, draw=mypink, thick] (5,3) circle (2pt);

\node at (-0.4,0) {\small $0$};
\node at (-0.4,1) {\small $1$};
\node at (-0.4,3) {\small $m$};
\node at (-0.7,4) {\small $m+1$};
\node at (0,-0.4) {\small $0$};
\node at (5,-0.4) {\small $n - 1$};
\node at (6,-0.4) {\small $n$};

\node[mypink] at (2.5,3.5) {\small $I$};
\node[myblue] at (2.5,1.5) {\small $J$};
\node at (5.5,3.5) {\small $\mathcal{T}$};
\node[mypink] at (3.5,3.5) {\small $\mathcal{T}'$};

\end{tikzpicture}
\end{center}

We want to find a $C_{p^{n}q^{m+1}}$-indexing set $K$ such that $K \Mod{p^nq^m} = J$ and $K\Mod{p^{n-1}q^{m+1}} = I$. Since $I$ and $J$ are fixed, and the transfer system induced by $K$ is saturated, the latter is fully determined by the data of whether $(n,m)\to(n,m+1)$ and $(n-1,m+1)\to(n,m+1)$ are admissible for $K$. Hence, $K$ realizes $\mathcal{T}$ if and only if the admissibility of these two relations in $\mathcal{T}$ and $K$ is the same. Indeed, all other relations except $(n-1,m)\to (n,m+1)$ are determined by $I$ and $J$, and this last relation is admissible if and only if all sides of the top-right square are admissible. If $K$ is built in this way, we have $K \Mod{p^iq^{m+1}} = I \Mod{p^iq^{m+1}}$ for all $0\leq i\leq n-1$, and $I$ satisfies $(\star)$, so $K$ satisfies $(\star)$ if and only if $K$ contains a non-zero multiple of $p^{n-1}q^{m+1}$.\\

Let us name the possibilities for the top-right square, with corners $(n-1,m+1)$, $(n,m+1)$, $(n,m)$, and $(n-1,m)$:
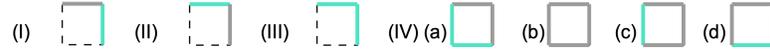
\begin{center}
\begin{tikzpicture}[scale=0.7, every node/.style={font=\footnotesize}]

\drawbox{0}  {(I)}  {G}{C}{D}{D}
\drawbox{2}  {(II)} {C}{G}{D}{D}
\drawbox{4}  {(III)} {C}{C}{D}{D}
\drawbox{6.5}  {(IV.a)} {G}{G}{C}{C}
\drawbox{8.5}  {(IV.b)} {G}{G}{G}{G}
\drawbox{10.5} {(IV.c)} {G}{G}{G}{C}
\drawbox{12.5} {(IV.d)} {G}{G}{C}{G}

\end{tikzpicture}
\end{center}

Red represents admissibility in $\mathcal{T}$, and gray not being admissible. Stability under restriction, transitivity, and saturation of $\mathcal{T}$ then determine the status of the dashed maps:
\begin{center}
\begin{tikzpicture}[scale=0.7, every node/.style={font=\footnotesize}]

\tikzset{
    mypinkline/.style={draw=mypink, line width=2pt},
    grayline/.style={draw=gray, line width=2pt}
}

\drawbox{0}  {(I)}  {G}{C}{G}{C}
\drawbox{2}  {(II)} {C}{G}{C}{G}
\drawbox{4}  {(III)} {C}{C}{C}{C}
\drawbox{6.5}  {(IV.a)} {G}{G}{C}{C}
\drawbox{8.5}  {(IV.b)} {G}{G}{G}{G}
\drawbox{10.5} {(IV.c)} {G}{G}{G}{C}
\drawbox{12.5} {(IV.d)} {G}{G}{C}{G}

\end{tikzpicture}
\end{center}

\begin{Notation}\label{Notation} The condition $I\Mod{p^{n-1}q^m} = J \Mod{p^{n-1}q^m}$ implies that for all $i\in I$ and $j\in J$ there exists $j_i\in J, i_j\in I, \gamma_i,\delta_j\in\Z$ with $i = j_i +\gamma_ip^{n-1}q^m$ and $j=i_j+\delta_jp^{n-1}q^m$. Since $p$ and $q$ are distinct primes, by Bezout's identity, there exist $u,v\in \Z$ with $uq+vp=1$. By Euclidean division, let $-u\gamma_i = r'_i p + r_i$ with $0\leq r_i < p$ and $-v\delta_j = s'_j q + s_j$ with $0\leq s_j < q$, for all $i\in I$ and $j\in J$. If $r_0 = 0$, we modify our choices in the following way. By assumption $J$ contains a non-zero multiple of $p^{n-1}q^m$, say $\alpha p^{n-1}q^m$ with $0<\alpha< p$. Choose $j_0 = \alpha p^{n-1}q^m$ and $\gamma_0 = -\alpha$. Since $p\nmid u$, we get $p \nmid -u\gamma_0$, so $r_0 \neq 0$. Similarly, if $s_0=0$, using the fact that $I$ contains a non-zero multiple of $p^{n-1}q^m$, we modify our choices of $i_0$ and $\delta_0$ to ensure $s_0\neq 0$. We fix a choice of such integers $j_i$, $i_j$, $\delta_j$, $\gamma_i$, $u$, $v$ (and thus $r_i$, $r'_i$, $s_j$, and $s'_j$) throughout the rest of the proof. 
\end{Notation}

\subsubsection*{\normalfont \framebox{Cases (I) and (III)}} If $(n,m)\to (n,m+1)$ is in $\mathcal{T}$, consider, in $C_{p^nq^{m+1}}$:
$$K := \pm \{ r_i p^{n-1}q^{m+1} + i + kp^nq^m \mid i \in I, k \in \Z\} \cup (\pm\{j + kp^nq^m \mid j \in J, k \in \Z\}).$$ Firstly, $K$ contains $p^nq^m$ (setting $j=0$, $k=1$). Taking $i=0$, $k=0$ in the definition of $K$, we have $r_0 p^{n-1}q^{m+1}\in K$. The choices made in Notation \ref{Notation} ensure $0<r_0<p$ and thus $(\star)$ holds. Moreover, we have:

\begin{itemize}[leftmargin=*]
    \itemsep0ex
    \item $0\in K$ (setting $j=0\in J$ and $k = 0$) and $-K\subseteq K$ by construction.
    
    \item $K \Mod{p^nq^m} = J$, since $K \Mod{p^nq^m} = \pm \{ r_i p^{n-1}q^{m+1} + i \mid i \in I\} \cup (\pm J)$ (the minus signs remain unchanged because the inverse$\mod{p^nq^{m+1}}$ of an element ${0\leq x < p^nq^{m+1}}$ is given by $p^nq^{m+1}-x$, but$\mod{p^nq^m}$ this is congruent to $-x \equiv p^nq^m-x$). Moreover, $\pm J = J$ because $J$ is an indexing set. Therefore, $J\subseteq K \Mod{p^nq^m}$, and for all $i\in I$ we have 
    \begin{align*}
        r_i p^{n-1}q^{m+1} + i &= -uq(\gamma_i p^{n-1}q^m) - r'_i p^nq^{m+1} + i  \\
        &\equiv (vp-1)(\gamma_i p^{n-1}q^m) + i \equiv v\gamma_i p^nq^m + j_i \equiv j_i \in J \Mod{p^nq^m}
    \end{align*}
    So the other inclusion holds as well.
    \item $K\Mod{p^{n-1}q^{m+1}}=I$: once more we simplify 
    $$K \Mod{p^{n-1}q^{m+1}} = \pm \{ i + kp^nq^m \mid i \in I, k \in \Z\} \cup (\pm\{j + kp^nq^m \mid j \in J, k \in \Z\}).$$
    Choosing $k=0$ we see that $I\subseteq K \Mod{p^{n-1}q^{m+1}}$. For the other inclusion, since we are in case (I) or (III), by Proposition \ref{Prop:IndexingSets}, $I$ is $p^{n-1}q^m$-translation invariant because $I$ admits $(n-1,m)\to (n-1,m+1)$, so in particular it is $p^nq^m$-translation invariant. Therefore, $i+kp^nq^m \in I \Mod{p^{n-1}q^{m+1}}\ \forall k\in \Z, i\in I$ and 
    $$j + kp^nq^m = i_j + \delta_j p^{n-1}q^m + kp^nq^m = i_j + p^{n-1}q^m (\delta_j + kp) \in I.$$
    Since $-I\subseteq I$, we conclude that $K\Mod{p^{n-1}q^{m+1}}\subseteq I$.
    \item Finally, $(n,m)\to (n,m+1)$ is in $K$: by Proposition \ref{Prop:IndexingSets}, it suffices to check that $K$ is $p^nq^m$-translation invariant, but this holds by construction.
\end{itemize}
This suffices to show that $K$ realizes $\mathcal{T}$: indeed, both $\mathcal{T}$ and the transfer system that $K$ realizes admits $(n,m)\to (n,m+1)$. But then, by saturation, the admissibility of $(n-1,m)\to (n,m)$ suffice for both of them to determine whether they are in situation (I) or (III), and the two transfer systems coincide there, since $J$ realizes the restriction of $\mathcal{T}$ and the reduction of $K$ equals $J$.

\subsubsection*{\normalfont \framebox{Case (II) (or (III))}} Assume now that $\mathcal{T}$ admits $(n-1,m+1)\to(n,m+1)$. The proof is symmetric to the previous case, this time consider (in $C_{p^nq^{m+1}}$):
\[
K := \pm \{ i + kp^{n-1}q^{m+1} \mid i \in I, k \in \Z\} \cup (\pm\{s_j p^nq^m + j + kp^{n-1}q^{m+1} \mid j \in J, k \in \Z\})
\]
Again, since $0<s_0<q$ by the choices made in Notation \ref{Notation}, $K$ contains a non-zero multiple of $p^nq^m$. Moreover, $K$ contains $p^{n-1}q^{m+1}$. As above, one easily checks that $0\in K$ and $-K\subseteq K$, and $K \Mod{p^{n-1}q^{m+1}} = I$, $K\Mod{p^n q^m}=J$ (in cases (II) and (III), $J$ is $p^{n-1}q^m$-translation invariant (modulo $p^{n}q^{m}$), so in particular it is $p^{n-1}q^{m+1}$-translation invariant). Finally, $K$ admits $(n-1,m+1)\to (n,m+1)$ because it is $p^{n-1}q^{m+1}$-translation invariant by construction. As above, this suffices to prove that $K$ realizes $\mathcal{T}$ because whether we are in case (II) or (III) only depends on the restriction of $\mathcal{T}$ to $[n-1]\times [m+1]$, corresponding to $I$, but $K$ extends $I$.

\subsubsection*{\normalfont \framebox{Case (IV)}} We distinguish subcases:\\

\textit{Case (a).} In this situation, $I$ and $J$ are both $p^{n-1}q^m$-translation invariant. In particular, they are the set of iterated translations by $p^{n-1}q^m$ of a common indexing set $L$ in $C_{p^{n-1}q^m}$, modulo $p^{n-1}q^{m+1}$ and $p^nq^m$ respectively. 
Assume $q>p$ to begin with. Consider
    \[
    K := \pm \{ \ell + kp^{n-1}q^m \mid \ell\in L, 0\leq k < 2q\} \subseteq C_{p^{n}q^{m+1}}.
    \]
    In particular, for $\ell = 0$ and $k=q$, $K$ contains $p^{n-1}q^{m+1}$, respectively $p^nq^m$ for $\ell=0$ and $k=p<2q$, which takes care of $(\star)$. By construction $K$ is an indexing set, and $K \Mod{p^{n-1}q^m} =\pm L = L$. Therefore, it suffices to show that $K$ is $p^{n-1}q^m$-translation invariant modulo ${p^{n-1}q^{m+1}}$, respectively $p^nq^m$, to show that it coincides with $I$, respectively $J$ there. And indeed, for all $\ell\in L$ and $0\leq k < 2q$, we have $\ell + kp^{n-1}q^m + p^{n-1}q^m = \ell + (k+1)p^{n-1}q^m$, which is in $K$ by definition if $k<2q-1$. 
    
    For $k=2q-1$, we find $\ell+2p^{n-1}q^{m+1} \equiv \ell \Mod{p^{n-1}q^{m+1}}$, and writing the Euclidean division $2q = s p + r$, with $1\leq r <p$ and $s\geq 1$ since $p<q$, we also have $\ell + 2qp^{n-1}q^{m} = \ell+ s p^nq^{m} + rp^{n-1}q^m \equiv \ell + rp^{n-1}q^m \Mod{p^nq^m}$ which lies in $K$ since $r<p<q$. 
    
    The argument for additive inverses is the same as the previous cases. It remains to check that $K$ is not $p^{n-1}q^{m+1}$-translation invariant modulo $p^nq^{m+1}$. Then, since we are in case $(a)$, by saturation, it means that it is not $p^nq^m$-translation invariant either. But $2 p^{n-1}q^{m+1} < p^nq^{m+1}$ (since $p>2$) is not contained in $K$: indeed, we have $0\leq \ell + kp^{n-1}q^m \leq p^{n-1}q^m - 1 + (2q-1)p^{n-1}q^m = 2p^{n-1}q^{m+1} - 1$ for all $\ell \in L$ and $0\leq k <2q$, and 
    \begin{align*}
        p^nq^{m+1} - (\ell+kp^{n-1}q^m) &> p^nq^{m+1} - 2p^{n-1}q^{m+1} = (p-2)(p^{n-1}q^{m+1}) \\
        &\geq 2p^{n-1}q^{m+1} \text{ since $p\geq 5$}
    \end{align*}
    Hence $K$ realizes $\mathcal{T}$. If $p>q$, the same proof applies by replacing $2q$ by $2p$ in the definition of $K$.\\
    
\textit{Case (b).} Since transfer systems are closed under restriction, any indexing set extending $I$ and $J$ suffices. Consider
$$K := \pm \{r_i p^{n-1}q^{m+1} + i \mid i\in I\} \cup (\pm \{s_j p^nq^m + j \mid j\in J\}) \subseteq C_{p^nq^{m+1}}.$$
Again, since $0<r_0<p$ and $0<s_0<q$, $K$ contains non-zero multiples of $p^{n-1}q^{m+1}$ and $p^nq^m$. Then, as before we have
\[
K \Mod{p^nq^m} = \pm \{r_i p^{n-1}q^{m+1} + i \mid i\in I\} \cup (\pm J)
\]
and $r_i p^{n-1}q^{m+1} + i \equiv j_i \in J \Mod{p^nq^ m}$, so $K \Mod{p^nq^m} = J$, and similarly $K \Mod{p^{n-1}q^{m+1}} = I$. Thus, $K$ realizes $\mathcal{T}$.\\

\textit{Case (c).}  We note first that any indexing set extending both $I$ and $J$ will not admit $(n-1,m+1)\to(n,m+1)$ by closure under restriction. It therefore suffices to find an indexing set $K \subseteq C_{p^nq^{m+1}}$, extending $I$ and $J$, not admitting $(n,m)\to (n,m+1)$. Consider $\tilde{K} := \pm \{ r_i p^{n-1}q^{m+1} + i \mid i \in I\} \cup (\pm J)$
and set $K := (\tilde{K} \setminus (\pm \{p^nq^m\})) \cup (\pm \{p^nq^m + \alpha up^{n-1}q^{m+1}, 2p^nq^m\})$
 where $\alpha p^{n-1}q^m \in J$ with $0<\alpha<p$ (exists by assumption).

Now, $\tilde{K}$ contains a non-zero multiple of $p^{n-1}q^{m+1}$, namely $r_0p^{n-1}q^{m+1}$, and it is contained in $K$ too (since it cannot be equal to $\pm p^nq^m$, otherwise $p^nq^m$ divides $r_0p^{n-1}q^{m+1}$, so $p \mid r_0q$, but this is impossible because $0< r_0 < p$), and $K$ contains $2p^nq^m \not\equiv 0$ (since $q>2$). Then $K\Mod{p^nq^m}$ clearly contains $J$, and $\alpha up^{n-1}q^{m+1}= \alpha (1-vp)p^{n-1}q^m \equiv \alpha p^{n-1}q^m \Mod{p^nq^m},$ which is contained in $J$ by hypothesis. Also, as before $r_i p^{n-1}q^{m+1} + i \equiv j_i \Mod{p^nq^m}$, so $K\Mod{p^nq^m} =J$. 

It is clear that $K\Mod{p^{n-1}q^{m+1}}$ contains $I$, and it is contained in $I$, indeed $p^nq^m + \alpha up^{n-1}q^{m+1}, 2p^nq^m \in I$ modulo $p^{n-1}q^{m+1}$ because $I$ is $p^{n-1}q^m$-translation invariant (case (c)); and for the same reason, for all $j\in J$, $j=i_j + \delta_j p^{n-1}q^m$ is in $I \Mod{p^{n-1}q^{m+1}}$. So $K \Mod{p^{n-1}q^{m+1}} = I$.
    
Finally, $K$ is not $p^nq^m$-translation invariant. Indeed, $p^nq^m \notin K$: we have $p^nq^{m+1}-2p^nq^m \neq p^nq^m$ since $q>3$. Furthermore, 
\[p^nq^m \not\equiv p^nq^m + \alpha up^{n-1}q^{m+1} \Mod{p^nq^{m+1}}\]
since $p\nmid \alpha u$, and finally $p^nq^m \not\equiv - p^nq^m - \alpha up^{n-1}q^{m+1} \Mod{p^nq^{m+1}}$ else $p^nq^{m+1}$ divides $ 2p^nq^m + \alpha u p^{n-1}q^{m+1}$ and $p$ would divide $\alpha u$. \\

\textit{Case (d).} As in the previous case, it suffices to find an indexing set $K$ extending both $I$ and $J$, that does not admit $(n-1,m+1)\to(n,m+1)$. The proof is done as in the previous case, with $\tilde{K} :=  (\pm I) \cup \pm \{ s_j p^{n}q^{m} + j \mid j \in J\}$
and $$K := \tilde{K} \setminus (\pm \{p^{n-1}q^{m+1}\}) \cup (\pm \{p^{n-1}q^{m+1} + \beta vp^{n}q^{m}, 2p^{n-1}q^{m+1}\})$$
where $\beta p^{n-1}q^m \in I$, with $0<\beta<q$, is our non-zero multiple of $p^{n-1}q^m$ contained in $I$.
\end{proof}

\section{Examples and comparison with MacBrough's approach\texorpdfstring{ in \cite{Macbrough}}{}}\label{Sub:example}
\addtocontents{toc}{\protect\setcounter{tocdepth}{1}}

The inductive proof of Theorem \ref{Prop:SatConjp^nq^m} in Section \ref{Sub:Proof}, being constructive and very explicit, can easily be turned into an algorithm to compute an explicit linear isometries operad associated to a given saturated transfer system on $C_{p^nq^m}$. Such an algorithm requires in principle $(n+1)(m+1)$ iterations of the ``outer loop'', namely $(n+1)(m+1)$ applications of Lemma \ref{Prop:InductionStepProof}, as we extend our indexing set ``one square at a time'' in the grid $[n] \times [m]$. However, extending the indexing set has a complexity which also depends on the primes $p$ and $q$ and the size of the indexing sets we are extending, so this is not an estimate of the complexity of the algorithm. We provide in this section a very basic example of explicit computations using our proof.

\subsection{Our proof applied to an explicit example}\label{Sub:explicitexample}

Consider the case of the group $C_{pq} = C_{35}$ with $p=5$, $q=7$. We consider the transfer system $\mathcal{T}$ on $[1]\times [1]$ with only non-trivial relation $(0,0)\to (1,0)$ (i.e.\ as case (IV)$(d)$ in Subsection \ref{Sub:cases}). 

In view of Lemma \ref{Prop:InductionStepProof} (with $n=1$ and $m=0$), we first want to find an indexing set $J \subseteq C_{p^1q^0} = C_5$ realizing the restriction of our transfer system to $[1]\times[0]$, namely the poset of subgroups of $C_p$ itself, and satisfying $(\star)$. An application of Lemma \ref{Prop:BaseCaseProof} (or direct observation) yields $J = C_5$ (this is actually the unique choice, by Proposition \ref{Prop:IndexingSets}$(ii)$). Also, the proof of Lemma \ref{Prop:BaseCaseProof} provides 
$$I=\{0,1,q-1\} = \{0,1,6\} \subseteq C_{p^0q^1}.$$

In Notation \ref{Notation}, we have to choose $u,v\in \Z$ with $uq+vp=1$. We take $v = 3$ and $u=-2$. For all $j\in J$, we also have to choose $i_j\in I, \delta_j\in\Z$ with $j=i_j+\delta_jp^{n-1}q^m$, and then by Euclidean division we write $-v\delta_j = s'_j q + s_j$ with $0\leq s_j < q$, for all $j\in J$. Here, we pick $\delta_j = j-1$ and $i_j=1$ for all $j\in J$ (since $p^{n-1}q^m = 1$) and thus
$$s_0 = 3, \qquad s_1=0,\qquad s_2=4, \qquad s_3=1,\qquad  s_4 = 5.$$

The proof of Lemma \ref{Prop:InductionStepProof} (case (IV)$(d)$) shows that, for the choice $\beta=1$, our transfer system $\mathcal{T}$ is realized by the indexing set:
$$K = (\tilde{K} \setminus \{q,pq-q\}) \ \cup \ \{q+vp,pq-q-vp,2q,pq-2q\} \subseteq C_{pq} = C_{35}$$
where $$ \tilde{K} = \{0, 1,q-1,pq-1,pq-q+1\} \cup \pm \{s_jp+j \mid j\in J\}.$$
More explicitly, in our case
$$\{q+vp,pq-q-vp,2q,pq-2q\} = \{22,13,14,21\}$$
and 
$$ \tilde{K} \setminus \{q,pq-q\} =(\{0,1,6,34,29\} \cup \{15,1,22,8,29,20,34,13,27,6\}) \ \setminus \ \{7,28\}.$$
In the end, we obtain the indexing set
$$K= \{0,1,6,8,13,14,15,20,21,22,27,29,34\} \subseteq C_{35}.$$
We obtain an explicit linear isometries operad using Definition \ref{Def:indexing}: the operad $\mathcal{L}(\mathcal{U}_K)$ for $\mathcal{U}_K = \bigoplus_{n\in \N}\bigoplus_{k\in K} \lambda_{35}(k)$.

\subsection{Comparison with MacBrough's approach}

The difficulties mentioned in the introduction preventing us from extending our approach to cyclic groups whose order has more than two prime divisors are elegantly circumvented by MacBrough. His approach is fundamentally different from ours in two ways. 

On the one hand, the main feature of his proof is that it does not depend on the description of a particular transfer system one would like to realize by a linear isometries operad. Rather, MacBrough defines abstract devices associated to a group (cyclic or not), called \emph{tight pairs}, whose existence witnesses that the group in question satisfies the saturation conjecture. Then, given a particular saturated transfer system, there is an algorithm to produce a linear isometries operad realizing it from this device. The algorithm can be written abstractly without using any particular features of the specific group or saturated transfer system under consideration. 

On the other hand, tight pairs are well-behaved under products of groups of coprime order; this is what allows MacBrough to prove the existence of tight pairs for all cyclic groups of order coprime to 6 by reducing to the case of cyclic $p$-groups $C_{p^n}$ (\cite[Lemmata 3.3 and 3.4]{Macbrough}). More precisely, there is a tensor product operation for tight pairs, such that the tensor product of a tight pair for a group $G$ and one for a group $G'$ with $(|G|,|G'|) = 1$ is a tight pair for $G\times G'$. This is a very efficient way around a big difficulty of dealing with (saturated) transfer systems on products of groups, namely that a priori they do not admit any general description in terms of (saturated) transfer systems on the factors. Since tight pairs do not have anything to do with a particular transfer system, but rather depend only on the group itself, they bypass this difficulty, as explained at the beginning of \cite[\S 3.2]{Macbrough}.\\

Therefore, in the case of a cyclic group $G$, MacBrough's proof can also in principle be made into an iterative algorithm, requiring, as a very rough estimate, at most $\sum_{H\leq G} |H|$ iterations of the ``outer loop''. Indeed, the proof is building compatible universes for each subgroup of $G$ (represented by the objects called \emph{diagrams} below), and adds at least one isotypic component to one of the universes at each step. There are thus at most $\sum_{H\leq G} |H|$ possible iterations before we reach the maximal universes. In the case $|G|=p^nq^m$, ours required in principle $(n+1)(m+1)$ applications of Lemma \ref{Prop:InductionStepProof}, as mentioned earlier. However, neither of these statements constitutes an estimate of the complexity of the corresponding algorithm, because each application of Lemma \ref{Prop:InductionStepProof}, respectively the operation of determining which components to add to the universe, also has a variable complexity.\\

For comparison, we repeat the example of Subsection \ref{Sub:explicitexample}, following MacBrough's proof instead. Recall that we considered the case $p=5$, $q=7$, $n=m=1$, and the saturated transfer system with only non-trivial relation $(0,0)\to (1,0)$. In the following, we always view $C_p \times C_q$ as $C_{pq}$ via the group isomorphism sending $[1]\in C_{pq}$ to $([1],[1]) \in C_p\times C_q$.

By \cite[Lemma 3.3]{Macbrough}, we have to construct tight pairs for the groups $C_p$ and $C_q$ and tensor them to obtain a tight pair for $C_{pq}$. 

\begin{Def}[\protect{\cite[Def.\ 3.1]{Macbrough}}] Let $G$ be a finite abelian group. A \emph{tight pair} for $G$ is a pair ($D$,$J$) where:
\begin{itemize}[itemsep = 0ex, leftmargin = *]
    \item $D = (D(H))_{H\leq G}$ is a collection of subsets $D(H) \subseteq \widehat{H}$ indexed by subgroups of $G$, where $\widehat{H}$ is the collection of isomorphism classes of (finite dimensional) complex irreducible representations of $H$.
    \item $J = (J^H_K)_{K\leq H\leq G}$ is a \emph{subinductor}, i.e.\ a collection of maps $J^H_K : \mathcal{P}(\widehat{K}) \to \mathcal{P}(\widehat{H})$ from the subsets of $\widehat{K}$ to those of $\widehat{H}$, indexed by intervals $K\leq H$ in the poset of subgroups of $G$, satisfying certain axioms (see \cite[Def.\ 2.2]{Macbrough}).
    \item $D$ and $J$ satisfy certain axioms, which we don't need to know explicitly to apply MacBrough's proof to our example. 
\end{itemize}
\end{Def}

Tight pairs for cyclic $p$-groups are explicitly constructed in \cite[Lemma 3.4]{Macbrough}. In our case, the proof gives the tight pair $(D,J)$ for $C_p$ defined as follows. The subinductor $J$ is defined by $J_{\ast}^{C_p}(\{\CC_\text{triv}\}) = \{\CC_{\text{triv}}\}$, $J_{\ast}^{C_p}(\emptyset) = \emptyset$, and the other maps in the subinductor are the identities. The diagram $D$ is given by $D(\ast) = \{\CC_\text{triv}\}$ and $D(C_p) = \{\CC_\text{triv}, \tau, \overline{\tau}\}$, where $\tau$ is an arbitrary element in $\widehat{C_p} \setminus\{\CC_\text{triv}\}$ which embeds into $\text{Ind}^{C_p}_\ast(\CC_\text{triv}) = \bigoplus_{V\in \widehat{C_p}} V$, where $\text{Ind}$ denotes the induction of representations. We may then choose $\tau = \lambda_p(j)$ for any $0<j< p$, and get $D(C_p) = \{\CC_\text{triv}, \lambda_p(j), \lambda_p(p-j)\}$.  For $C_q$, a tight pair $(D',J')$ is defined in the same way, replacing every instance of $p$ by $q$ in the above, and choosing $0 < k<q$ instead of $0<j<p$.

Then, by \cite[Lemma 3.3]{Macbrough}, $(D'', J'') := (D\otimes D',J\otimes J')$ is a tight pair for $C_{pq}$, where:
    \begin{align*}
        D''(\ast \times \ast) &= \{V \otimes W \mid V\in D(\ast), W\in D'(\ast)\} = \{\CC_\text{triv}\} \\
        D''(C_p \times \ast) &= \{V \otimes W \mid V\in D(C_p), W\in D'(\ast)\} = \{\CC_\text{triv}, \lambda_{p}(j), \lambda_{p}(p-j)\}\\
        D''(\ast \times C_q) &= \{V \otimes W \mid V\in D(\ast), W\in D'(C_q)\} = \{\CC_\text{triv}, \lambda_{q}(k), \lambda_{q}(q-k)\}\\
        D''(C_p \times C_q) &= \{V \otimes W \mid V\in D(C_p), W\in D'(C_q)\} \\
        &= \{\CC_\text{triv}, \lambda_{pq}(jq), \lambda_{pq}(pk), \lambda_{pq}((p-j)q), \lambda_{pq}(p(q-k)),\lambda_{pq}(pk+jq),\\
        &\qquad\lambda_{pq}(pq - pk-jq), \lambda_{pq}(pk+(p-j)q), \lambda_{pq}(p(q-k) + jq)\},
    \end{align*}
and 
\begin{align*}
{J''}_\ast^{C_p} &= {J}_\ast^{C_p}\\
{J''}_\ast^{C_q} &= {J'}_\ast^{C_q}\\
{J''}_\ast^{C_{pq}} &= {J''}_{C_p}^{C_{pq}} \circ {J''}_\ast^{C_{p}},
\end{align*}
where ${J''}_{C_p}^{C_{pq}}$ is the unique union-preserving extension of ${J''}_{C_p}^{C_{pq}}(\{\lambda_p(i)\}) = \{\lambda_{pq}(iq)\}$. Finally, ${J''}_{C_q}^{C_{pq}}$ is defined similarly. The other maps in the subinductor are the identities.\\

Applying the proof of \cite[Thm 2.12]{Macbrough}, we have to compute successive extension of the diagram $D$ until the algorithm converges, i.e.\ until these extensions do not add any isotypic component. We first have to compute the diagram $D''_1$ defined by 
$$\forall H\leq G, \ D''_1(H) = \bigcup_{(K\to H)\in\mathcal{T}} \text{Ind}^H_K(D''(K)),$$ 
where $\mathcal{T}$ is our saturated transfer system, and for any $E\subseteq \widehat{K}$, we define
$$\text{Ind}^H_K(E) = \{V \in \widehat{H} \mid \exists W\in E, V \subseteq \text{Ind}^H_K(W)\}.$$
Then $D''_1(H) = D''(H)$ for all subgroups $H$ except for 
$$D''_1(C_p \times \ast) = D''(C_p\times \ast) \cup \text{Ind}_{\ast \times \ast}^{C_p\times \ast}(D''(\ast\times\ast))=\{\lambda_p(i) \mid 0\leq i\leq p-1\} = \widehat{C_p}.$$ 

The next iteration consists in computing the diagram $D''_2$ defined by 
$$\forall H\leq G,\ D''_2(H) = \bigcup_{K\leq H} {J''}^H_K(D''_1(K)).$$
Again, $D''_2(H) = D''_1(H)$ for all subgroups $H$ except for 
\begin{align*}
    D''_2(C_p \times C_q) &=  J^{C_{pq}}_{C_q}(D''_1(\ast \times C_q)) \cup J^{C_{pq}}_{C_p}(D''_1(C_p \times \ast))  \cup D''_1(C_p\times C_q)\\
    &= \{\lambda_{pq}(iq)\mid 0\leq i\leq p-1\} \cup D''_1(C_p\times C_q).
\end{align*}

We then see that the next extensions do not change anything to the diagram, so the algorithm stops. We obtain the universe whose set of isotypic components is $D''_2(C_{pq}) = \{0,2,5,7,12,14,21,23,28,30,33\}$ in the case $k=1$, $j=1$ for example. Note that the result may never be the exact same indexing set as we obtained in Subsection \ref{Sub:explicitexample} for any choice of $j$ and $k$, since there will always appear $\lambda_{pq}(iq) \in D''_2(C_p\times C_q)$ for $0\leq i<q$, whereas the indexing set we found before does not contain $\lambda_{pq}(q) = \lambda_{35}(7)$. However, we also had to make choices in our construction.\\

The approach of MacBrough can also produce the same indexing set as ours. Here is an example. Consider the case of the group $C_{p^2}$ with $p=5$, and the saturated transfer system with only non-trivial relation $1\to 2$. Then our construction yields the indexing set $\{ \pm (k + ip) \mid k \in \{0,1,p-1\}, 0\leq i < p \}$, whereas MacBrough's yields, for the choice of tight pair $(D,J)$ with
\begin{align*}
    J^{C_{p^2}}_{C_p}(\{\lambda_p(i)\}) &= \{\lambda_{p^2}(i), \lambda_{p^2}(p^2-p+i)\}\\
    D(\ast) &= \{\CC_\text{triv}\}\\
    D(C_p) &= \{\CC_\text{triv}, \lambda_p(1), \lambda_p(p-1)\}\\
    D(C_{p^2}) &= \{\CC_\text{triv}, \lambda_{p^2}(2p), \lambda_{p^2}(p^2-2p)\}
\end{align*}
(and the remainder of the subinductor is induced by requiring preservation of unions) the indexing set $\{0,1,4,5,6,9,10,11,14,15,16,19,20,21,24\}$, which agrees with ours.

\bibliographystyle{alphaurl}
\bibliography{biblio.bib}

\end{document}